\documentclass[numbers,sort&compress]{elsarticle}
\usepackage{amsmath}
\usepackage{amsthm}
\usepackage{amssymb}
\usepackage{paralist}
\usepackage{MnSymbol}

\usepackage{xcolor}

\theoremstyle{definition} 
\theoremstyle{definition} \newtheorem{Theo}{Theorem}
\theoremstyle{definition} 
\theoremstyle{definition} 
\theoremstyle{definition} 
\theoremstyle{definition} 
\theoremstyle{definition} \newtheorem{Exa}{Example}
\theoremstyle{definition} 

\newcommand{\ud}{\, \mathrm{d}}

\allowdisplaybreaks[4]

\begin{document}
\begin{frontmatter}

\title{Rules and Algorithms\\ for Objective Construction of Fuzzy Sets}

\author[add1]{Lei Zhou\corref{cor1}}
\ead{zhoulei137@swust.edu.cn; zhoul137@mail.ustc.edu.cn}
\cortext[cor1]{Corresponding author}
\address[add1]{School of mathematics and physics, Southwest University of Science and Technology, 621010, Mianyang, China}

%


\begin{abstract}
This paper aims to present rigorous methodologies for constructing new fuzzy sets from existing fuzzy or classical sets, defined within the framework of a finite universe's superstructure. The paper introduces rules for assigning membership functions to these new fuzzy sets, resulting in two significant findings. Firstly, the extension of the classical property concerning the cardinality of a power set to the fuzzy context is demonstrated. Specifically, the scalar cardinality of a fuzzy set $\tilde{B}$, defined on the power set of a finite universe of a fuzzy set $\tilde{A}$, adheres to the relation $\text{card}(\tilde{B}) = 2^{\text{card}(\tilde{A})}$. Secondly, the novel algorithms enable the objective achievement and representation of arbitrary membership values through specific binary sequences.
\end{abstract}


\end{frontmatter}

\section{Introduction}
In classical set theory, an element is either a member of a set or not, represented by a discrete indicator function, or characteristic function, $\mathcal{I}_A(x): X \to \{0,1\}$, which is defined as:
\begin{equation}
\mathcal{I}_A(x) = \begin{cases}
1 & \text{if } x \in A, \\
0 & \text{if } x \notin A.
\end{cases}
\end{equation}

In 1965, Zadeh introduced the concept of a fuzzy set, which generalizes the classical notion of a set \cite{Zadeh1965}. In fuzzy set theory, the characteristic function is replaced by a membership function, $\mu_A(x): U \to [0,1]$, which indicates the degree to which an element $x$ belongs to set $A$ on a continuous scale between 0 and 1.

The membership function is a pivotal concept in fuzzy set theory, uniquely characterizing any fuzzy set. Core concepts such as the support of a fuzzy set, alpha-level set, convex fuzzy set, and set-theoretic operations (intersection, union, complement) are all defined based on membership functions \cite{Zimmermann2001}. Therefore, specifying the membership function is crucial in both the theoretical and practical aspects of fuzzy sets. Once the membership function is established, the corresponding fuzzy set is determined. Various methods have been developed to derive reasonable membership functions, including piecewise linear approaches like the widely used triangular and trapezoidal membership functions \cite{Pedrycz1994, Wang2015, Nguyen2019}. In practice, appropriate membership functions are often selected based on the subjective knowledge or perception of the observer \cite{Liu2009}, leading many researchers to consider the assignment of membership functions as inherently subjective \cite{Kandel1978, Singpurwalla2004, Ross2010, Belohlavek2011, Nguyen2019}. This subjectivity has motivated the development of objective methods for assigning membership functions that generate fuzzy sets impartially. Our proposed methods are designed to be context-independent, applicable universally without reliance on individual perceptions. These methods ensure that membership functions are derived from objective computations, thereby extending the applicability of fuzzy sets to a broader range of scenarios without introducing subjective uncertainty.

Our methods offer an advantage of objectivity by constructing new fuzzy sets from existing ones, drawing inspiration from classical set theory, specifically the ZFC axiomatic system \cite{Jech2003}. ZFC, which stands for Zermelo–Fraenkel set theory with the axiom of Choice, provides a framework for constructing new sets based on existing ones. For instance, the axiom of pairing states that given any two sets $X$ and $Y$, there exists a new set $\{X, Y\}$ containing exactly $X$ and $Y$. In the following sections, we explore the ZFC axiomatic system from a fuzzy perspective, aiming to construct new fuzzy sets from both existing fuzzy sets and classical sets. This approach treats some ZFC axioms as special cases in the fuzzy setting, considering classical sets as fuzzy sets with membership functions equal to 0 or 1.

To achieve this, we focus on the concept of cardinality, a measure of a set's size, counting the number of elements in the set. For classical sets, the cardinality of set $A$ is denoted as $|A|$ or $\text{card}(A)$. A finite set has a natural number as its cardinality, while a set with the same cardinality as the set of all natural numbers is countably infinite, denoted as $\text{card}(A) = \aleph_0$. Sets with greater cardinality are uncountable.

The theory of cardinality for fuzzy sets is a complex aspect of fuzzy set theory \cite{Wygralak2001}. Various definitions have been proposed and studied \cite{Dubois1990a, Dubois1990b, Delgado2000, Chamorro2014, Delgado2014, Michal2016, Michal2020}. The concept of scalar cardinality for fuzzy sets, introduced by De Luca and Termini, aids in understanding the measurement of information for fuzzy sets \cite{Luca1972}. In the 1980s, researchers debated whether the cardinality of a fuzzy set should be a precise real number or a fuzzy integer \cite{Zadeh1979, Gottwald1980, Blanchard1982, Wygralak1983}. Dubois and Prade proposed a unified framework \cite{Dubois1985}. This paper focuses solely on scalar cardinality.

For a fuzzy set $\tilde{A}$ defined on a universe $U$ with a membership function $\mu_{\tilde{A}}: U \to [0,1]$, the scalar cardinality is the sum of the membership values of all elements in $U$, given by
\begin{equation}
\text{card}(\tilde{A}) = \sum_{x \in U} \mu_{\tilde{A}}(x).
\end{equation}

Additionally, we consider the power set and its representation. The power set of $A$, denoted $\mathcal{P}(A)$ or $2^A$, includes all subsets of $A$. If $A$ is a finite set with $n$ elements, the cardinality of $\mathcal{P}(A)$ is $2^n$.

The superstructure over a set is another important concept. In fuzzy set theory, the universe of discourse or universe refers to the reference set. The superstructure over a universe $X$ is defined recursively: $S_0(X) = X$, $S_{n+1}(X) = S_n(X) \cup \mathcal{P}(S_n(X))$, and the superstructure over $X$, denoted $S(X)$, is
\begin{equation}
S(X) := \bigcup_{j=0}^{\infty} S_j(X).
\end{equation}

The remainder of this paper is structured as follows: Section 2 provides fuzzy interpretations of some ZFC axioms and introduces four rules for constructing new fuzzy sets from existing ones. We prove an important theorem related to power sets in the fuzzy setting. Section 3 introduces two additional rules for constructing new fuzzy sets from classical sets and proves a theorem on representing and generating any membership value using our rules. Section 4 includes numerical examples, and Section 5 concludes the paper.

\section{Constructing Fuzzy Sets from Existing Ones}

The ZFC set theory encompasses nine axioms foundational to classical set theory. Classical sets are essentially a subset of fuzzy sets, characterized by a membership function that exclusively takes values in the set $\{0, 1\}$. Within this framework, it is possible to reinterpret specific axioms of ZFC through the lens of fuzzy set theory. These reinterpretations establish criteria for the methodologies introduced in this paper for constructing new fuzzy sets, ensuring that these methodologies remain consistent with classical set theory. The ensuing paragraphs provide a detailed exploration of these reinterpretations and their implications.

\begin{enumerate}
\item The Axiom of Pairing asserts that for any two sets $A$ and $B$, there exists a set $\lbrace A, B\rbrace$ which contains exactly $A$ and $B$. The fuzzy set interpretation of this axiom stipulates that given two fuzzy sets $\tilde{A}$ defined on $X$ and $\tilde{B}$ defined on $Y$, such that for all $x \in X$ and $y \in Y$, we have $\mu_{\tilde{A}}(x) = 1$ and $\mu_{\tilde{B}}(y) = 1$, there exists a fuzzy set $\tilde{C}$ defined on $\lbrace X, Y \rbrace$ such that
\begin{equation}\label{pairing}
\mu_{\tilde{C}}(X) = \mu_{\tilde{C}}(Y) = 1.
\end{equation}
This indicates that the membership values of $\tilde{C}$ for the sets $X$ and $Y$ are both equal to 1, signifying that $\tilde{C}$ encompasses both $X$ and $Y$.

\item The Axiom of Union posits that the union of the elements in a set exists. For any set $X$, there exists a set $Y = \bigcup X$. In the context of fuzzy set theory, this axiom can be interpreted as follows: given a fuzzy set $\tilde{A}$ defined on $X$, with $\mu_{\tilde{A}}(x) = 1$ for every $x \in X$, there exists a fuzzy set $\tilde{B}$ defined on $Y$ such that
\begin{equation}\label{union}
\mu_{\tilde{B}}(y) = 1
\end{equation}
for any $y \in Y$, where $y \in x$ for some $x \in X$.

\item The Axiom of Power Set asserts that for any set $X$, there exists a set $\mathcal{P}(X)$ which contains all the subsets of $X$. In fuzzy set theory, this axiom can be interpreted to mean that if we have a fuzzy set $\tilde{A}$ defined on $X$, with $\mu_{\tilde{A}}(x) = 1$ for every $x \in X$, then there exists a fuzzy set $\tilde{B}$ defined on $Y = \mathcal{P}(X)$ such that
\begin{equation}\label{power}
\mu_{\tilde{B}}(y) = 1
\end{equation}
for every $y \in Y$, ensuring that every $x \in X$ is an element of at least one $y \in \mathcal{P}(X)$.

\end{enumerate}

In this section, we delineate the methodology for constructing new fuzzy sets utilizing principles analogous to those in ZFC set theory. Let $\tilde{A}$ be a fuzzy set defined on a universe $X$, and let $Y$ be another universe. We can construct a new fuzzy set, denoted by $\tilde{B}$, defined on $Y$ by adhering to the following construction rules:

\begin{enumerate}[{\textbf{Rule}} 1.]
\item The set $Y$ must be an element of the superstructure over $X$, denoted as $S(X)$, implying that $Y$ can be any classical set contained within $S(X)$.
\item The membership value $\mu_{\tilde{B}}(y)$ is defined to equal $\mu_{\tilde{A}}(x)$ if $y = x$; however, the converse does not necessarily hold.
\item If $Y$ includes the empty set, the membership function value at the empty set is always equal to 1: $\mu_{\tilde{B}}(\varnothing) \equiv 1$.
\item The membership function of the fuzzy set $\tilde{B}$ is given by
\begin{equation}\label{rule4}
\mu_{\tilde{B}}(y) = \prod_{x \in y} \left(2^{\mu_{\tilde{A}}(x)} - 1\right),
\end{equation}
where $y$ is any element contained within $Y$.
\end{enumerate}

We can interpret the rules described above as follows:

Rule 1 delineates the permissible universe for constructing a new fuzzy set from an existing one. For example, if the fuzzy set $\tilde{A}$ is defined on the universe $X = \{x_1, x_2\}$, we can construct another fuzzy set $\tilde{B}$ from $\tilde{A}$ on the universe $Y = \{\varnothing, x_1, \{x_2\}, \{x_1, \{x_1, x_2\}\}\}$. This is feasible because the element $x_1$ belongs to $S_0(X)$, and the elements $\varnothing, \{x_2\}, \{x_1, x_2\}$, and $\{x_1, \{x_1, x_2\}\}$ belong to $S_1(X)$ and $S_2(X)$, respectively.

Rule 2 asserts that the membership value remains unchanged for the same element when using this construction method.

Rule 3 is justified because the empty set, being the only set devoid of elements, is a subset of every set and cannot be partitioned.

Rule 4 provides the formula for calculating the new membership values. It is evident that Rule 3 is a special case of Rule 4, as it represents an empty product. Even if the universe of the new fuzzy set is complex, we can iteratively apply formula \eqref{rule4} until the information of the original fuzzy set is directly substituted into this formula.

Next, we will demonstrate the compatibility of the aforementioned rules with certain axioms of ZFC, as previously mentioned.

The fuzzy interpretation of the axiom of pairing, as expressed by formula \eqref{pairing}, can be derived using the rules presented above. Specifically, we have:
\begin{equation}
\mu_{\tilde{C}}(X) = \prod_{x \in X} \left(2^{\mu_{\tilde{A}}(x)} - 1\right) = \prod_{x \in X} \left(2^1 - 1\right) = 1,
\end{equation}
and similarly,
\begin{equation}
\mu_{\tilde{C}}(Y) = \prod_{y \in Y} \left(2^{\mu_{\tilde{B}}(y)} - 1\right) = \prod_{y \in Y} \left(2^1 - 1\right) = 1.
\end{equation}

Likewise, we can derive the fuzzy interpretation of the axiom of power set, as given by formula \eqref{power}:
\begin{equation}
\mu_{\tilde{B}}(y) = \left\lbrace
\begin{aligned}
&1,\quad y = \varnothing;\\[2mm]
&\prod_{x \in y}\left(2^{\mu_{\tilde{A}}(x)} - 1\right) = \prod_{x \in y}\left(2^1 - 1\right) = 1,\quad y \neq \varnothing.
\end{aligned}
\right.
\end{equation}

The derivation of the axiom of union differs slightly from the previous two axioms. In this case, we utilize the membership function of the fuzzy set $\tilde{A}$, as provided by formula \eqref{union}. Specifically, $\mu_{\tilde{A}}(x) = 1$ holds if and only if for every $y \in x$, we have $\mu_{\tilde{B}}(y) = 1$.

Formula (\ref{rule4}) is consistent with classical set theory, which asserts that an element either belongs to a set or does not. If all $x$ belong to a classical set, denoted as $\tilde{A}$, i.e., $\mu_{\tilde{A}}(x) \equiv 1$ for all $x$, then formula (\ref{rule4}) implies that $\mu_{\tilde{B}}(y) \equiv 1$ for any $y$. Conversely, if $\mu_{\tilde{A}}(x) = 0$ for some $x$, then formula (\ref{rule4}) implies that $\mu_{\tilde{B}}(y) = 0$ for any expression of $y$ that contains $x$. This outcome is logical, as a positive membership value cannot be derived from null.

\begin{Theo}\label{cardinality-power-set}
Let $\tilde A$ be a fuzzy set defined on a finite set $X$, where each element $x \in X$ is associated with its membership degree $\mu_{\tilde A}(x)$. Construct a fuzzy set $\tilde B$ from $\tilde A$, defined on the power set of $X$, such that
\begin{equation}
\tilde B=\left\{\big(y, \mu_{\tilde B}(y)\big)\middle| y \in \mathcal{P}(X) \right\}.
\end{equation}
The cardinality of the fuzzy set $\tilde B$ is then given by:
\begin{equation}\label{cardinality-powerset}
\text{card}(\tilde B)=2^{\text{card}(\tilde A)}.
\end{equation}
\end{Theo}

\begin{proof}
Consider $X_n=\{x_1, \dots, x_n\}$ as a finite set, and let $\tilde A_n$ be a fuzzy set defined on $X_n$, where $n$ is the cardinality of $X_n$. The proof will proceed by induction on $n$.

For the base case, let $n=1$. Here, $X_1=\{x_1\}$ is a single-element set, thus its power set is $\mathcal{P}(X_1)=\{\varnothing, \{x_1\}\}$. We can construct a fuzzy set $\tilde B_1$ on $\mathcal{P}(X_1)$ as follows:
\begin{equation}
\tilde B_1=\displaystyle\frac{1}{ \varnothing }+\frac{ \mu_{\tilde B_1}(\{x_1\})}{ \{x_1\} }.
\end{equation}
The cardinality of the fuzzy set $\tilde B_1$ is therefore
\begin{equation}
\text{card}(\tilde B_1)=1+ \mu_{\tilde B_1}(\{x_1\})=1+2^{ \mu_{\tilde A_1}(x_1)}-1=2^{ \mu_{\tilde A_1}(x_1)}=2^{\text{card}({\tilde A_1})}.
\end{equation}

For the induction step, assume that for a fuzzy set $\tilde B_k$ defined on the power set of $X_k$, the relation $\text{card}(\tilde B_k)=2^{\text{card}({\tilde A_k})}$ holds for $n=k$. We need to show that
\begin{equation}
\text{card}(\tilde B_{k+1})=2^{\text{card}({\tilde A_{k+1}})}.
\end{equation}

Given $X_{k+1}=X_k\cup \{x_{k+1}\}$, we have
\begin{equation}
\begin{aligned}
&\mathcal{P}(X_{k+1})\\
=&\mathcal{P}(X_k)\cup \{\{x_{k+1}\}\}\\
&\cup\{\{x_1,x_{k+1}\},\{x_2,x_{k+1}\},\cdots,\{x_k,x_{k+1}\}\}\\
&\cup\{\{x_1,x_2,x_{k+1}\},\{x_1,x_3,x_{k+1}\},\cdots,\{x_{k-1},x_k,x_{k+1}\}\}\\
&\cup\cdots\\
&\cup\{\{x_1,x_2,\cdots,x_k,x_{k+1}\}\}.
\end{aligned}
\end{equation}

Thus, the cardinality of the fuzzy set $\tilde B_{k+1}$ is computed as follows:
\begin{align*}
&\text{card}(\tilde B_{k+1})\\
=&2^{\text{card}({\tilde A_k})}+2^{\mu_{\tilde A_{k+1}}(x_{k+1})}-1\\
+\!&\left(2^{\mu_{\tilde A_{k+1}}(x_1)} \!-\!1 \right)\!\!\left(2^{\mu_{\tilde A_{k+1}}(x_{k+1})}\!-\!1 \right)\!+\!\left(2^{\mu_{\tilde A_{k+1}}(x_2)} \!-\!1 \right)\!\!\left(2^{\mu_{\tilde A_{k+1}}(x_{k+1})}\!-\!1 \right)\\
+& \cdots+\left(2^{\mu_{\tilde A_{k+1}}(x_k)} \!-\!1 \right)\!\!\left(2^{\mu_{\tilde A_{k+1}}(x_{k+1})}\!-\!1 \right)\\
+&\left(2^{\mu_{\tilde A_{k+1}}(x_1)} \!-\!1 \right)\!\!\left(2^{\mu_{\tilde A_{k+1}}(x_2)} \!-\!1 \right)\!\!\left(2^{\mu_{\tilde A_{k+1}}(x_{k+1})}\!-\!1 \right)+\cdots\\
+&\left(2^{\mu_{\tilde A_{k+1}}(x_{k-1})} \!-\!1 \right)\!\!\left(2^{\mu_{\tilde A_{k+1}}(x_k)} \!-\!1 \right)\!\!\left(2^{\mu_{\tilde A_{k+1}}(x_{k+1})}\!-\!1 \right)+\cdots\\
+&\left(2^{\mu_{\tilde A_{k+1}}(x_1)} \!-\!1 \right)\!\!\left(2^{\mu_{\tilde A_{k+1}}(x_2)} \!-\!1 \right)\cdots\left(2^{\mu_{\tilde A_{k+1}}(x_k)} \!-\!1 \right)\!\!\left(2^{\mu_{\tilde A_{k+1}}(x_{k+1})}\!-\!1 \right)\\
=&2^{\text{card}({\tilde A_k})} +   \text{card}(\tilde B_k) \left( 2^{\mu_{\tilde A_{k+1}}(x_{k+1})}-1\right)\\
=&2^{\text{card}({\tilde A_k})}2^{\mu_{\tilde A_{k+1}}(x_{k+1})}=2^{\text{card}({\tilde A_{k+1}})}.\\
\end{align*}
This completes the proof.
\end{proof}

\section{Constructing Fuzzy Sets from Classical Sets}

In this section, we address a pivotal issue in the domain of fuzzy set theory: the objective determination of arbitrary membership values. Elements within the superstructure over a singleton set are instrumental in this context. For notational convenience, a single element $x$ and its corresponding singleton set $\lbrace x \rbrace$ are denoted as $\{x \}^{(0)}$ and $\{x \}^{(1)}$, respectively. Furthermore, if we consider a classical set $\lbrace y \rbrace$ and set $x=\lbrace y \rbrace$, the element $y$ is represented by $\lbrace x \rbrace^{(-1)}$.

More generally, an element $x$ in nested braces is defined as $\{x \}^{(n)}:=\{\{x\}^{(n-1)}\}$ and $\{x \}^{(-n)}:=\{\{x\}^{(1-n)}\}^{(-1)}$, where $n\geq 1$. For example, $\{x \}^{(3)}$ is an abbreviation for $\{\{\{x \}\}\}$, and $\{x \}^{(-3)}$ denotes $\{\{\{x \}^{(-1)}\}^{(-1)}\}^{(-1)}$. Clearly, by setting $y=\{x \}^{(-3)}$, we derive that $x=\{\{\{{y }\}\}\}$.\\

We present a theorem pertaining to sets with nested braces:
\begin{Theo}
In the construction of new fuzzy sets, the following three expressions are interchangeable for any element $x$ in a classical set, where $m$ and $n$ are integers:
\begin{equation}
\left\lbrace\lbrace x \rbrace^{(m)}\right\rbrace^{(n)}, \quad \left\lbrace\lbrace x \rbrace^{(n)}\right\rbrace^{(m)},\quad \lbrace x \rbrace^{(m+n)}.
\end{equation}
\end{Theo}

\begin{proof}
The statement is evident when $m$ and $n$ share the same sign; thus, we focus on the scenario where $m$ and $n$ have opposite signs. Without loss of generality, assume $m > 0$ and $n < 0$. Let $\tilde A$ be a fuzzy set defined on the universe $X$, with $x \in X$ and $\mu_{\tilde A}(x) = u \in [0, 1]$. A new fuzzy set, $\tilde B$, can be constructed from $\tilde A$ by including $\lbrace\lbrace x \rbrace^{(m)}\rbrace^{(n)}$, $\lbrace\lbrace x \rbrace^{(n)}\rbrace^{(m)}$, and $\lbrace x \rbrace^{(m+n)}$ in its universe.

By applying Rule 4 and Rule 5, we derive the following equations:

\begin{align*}
&\mu_{\tilde B}\left\lbrace\lbrace x \rbrace^{(m)}\right\rbrace^{(n)}\\
=&\underbrace {\log_2\left(\log_2\left(\cdots\log_2\left(\left(\underbrace {2^{2^{^{\udots^{2^u-1}}}-1}-1}_{m{\text{ times}}\;2^{^{(\cdot)}}-1}\right)+1 \right)\cdots+1\right)+1\right)}_{n{\text{ times}} \log_2(\cdot)+1}\\
=&\left\lbrace 
\begin{aligned}
&u,\quad \text{if}\;\;m+n=0;\\[2mm]
&\underbrace {2^{2^{^{\udots^{2^u-1}}}-1}-1}_{m+n{\text{ times}}\;2^{^{(\cdot)}}-1},\quad \text{if}\;\;m+n>0;\\[2mm]
&\underbrace {\log_2\left(\log_2\left(\cdots\log_2(u+1)\cdots+1\right)+1\right)}_{|m+n|{\text{ times}} \log_2(\cdot)+1},\quad \text{if}\;\;m+n<0
\end{aligned}
 \right.
\end{align*}
and
\begin{align*}
&\mu_{\tilde B}\left\lbrace\lbrace x \rbrace^{(n)}\right\rbrace^{(m)}\\
=&\underbrace {2^{\displaystyle 2^{^{\udots^{\displaystyle 2^{\displaystyle\underbrace {\log_2\left(\log_2\left(\cdots\log_2\left(u+1 \right)\cdots+1\right)+1\right)}_{n{\text{ times}} \log_2(\cdot)+1}}-1}}}-1}-1}_{m{\text{ times}}\;2^{^{(\cdot)}}-1}  \\
=&\left\lbrace 
\begin{aligned}
&u,\quad \text{if}\;\;m+n=0;\\[2mm]
&\underbrace {2^{2^{^{\udots^{2^u-1}}}-1}-1}_{m+n{\text{ times}}\;2^{^{(\cdot)}}-1},\quad \text{if}\;\;m+n>0;\\[2mm]
&\underbrace {\log_2\left(\log_2\left(\cdots\log_2(u+1)\cdots+1\right)+1\right)}_{|m+n|{\text{ times}} \log_2(\cdot)+1},\quad \text{if}\;\;m+n<0.
\end{aligned}
 \right.
\end{align*}

In addition, Rule 2 can be applied to yield:
\begin{align*}
&\mu_{\tilde B}\left(\lbrace x \rbrace^{(m+n)}\right)\\
=&\left\lbrace
\begin{aligned}
&\mu_{\tilde B}\left(\lbrace x \rbrace^{(0)}\right)=\mu_{\tilde B}(x)=\mu_{\tilde A}(x)=u,\quad \text{if}\;\;m=n;\\[2mm]
&\underbrace {2^{2^{^{\udots^{2^u-1}}}-1}-1}_{m+n{\text{ times}}\;2^{^{(\cdot)}}-1},\quad \text{if}\;\;m+n>0;\\[2mm]
&\underbrace {\log_2\left(\log_2\left(\cdots\log_2(u+1)\cdots+1\right)+1\right)}_{|m+n|{\text{ times}} \log_2(\cdot)+1},\quad \text{if}\;\;m+n<0.
\end{aligned}
\right.
\end{align*}
These results demonstrate the substitution of the three expressions.
\end{proof}

We now delve into the methodology of constructing fuzzy sets from classical sets. To begin, consider a binary sequence composed of 1's and 0's, denoted as:
\begin{equation}\label{a-binary-sequence}
\mathbf{a}=(a_{m^*},a_{m^*+1},a_{m^*+2},\dots), m^*\leq 0.
\end{equation}

Here, we assume $a_{m^*} \equiv 1$ and $a_0 \equiv 1$, with other elements of $\mathbf{a}$ taking values either $0$ or $1$. For each $k \geq m^*$, we construct a classical set $X_{\mathbf{a}} \in S({x})$ based on $\mathbf{a}$: if $a_k = 1$, then the set $\{x\}^{(k)}$ is included in $X_{\mathbf{a}}$, whereas if $a_k = 0$, it is not. Notably, $X_{\mathbf{a}}$ always includes the element $x$ since $a_0 \equiv 1$. To simplify notation, we replace $a_0$ with a vertical line and use finite binary sequences to represent sequences with infinite 0's. For example, $X_{(1,0\; |\; 1,0,1,1)}$ denotes the set:
\begin{equation}
\left\{\{x\}^{(-2)},x, \{x\}^{(1)},\{x\}^{(3)},\{x\}^{(4)} \right\}.
\end{equation}

To construct fuzzy sets from classical sets, adherence to an additional important rule is required.\\

\noindent \textbf{Rule 5.} A fuzzy set $\tilde{A}$ can be constructed from any singleton set $\lbrace x \rbrace$ and any binary sequence $\mathbf{a}$ generated by \eqref{a-binary-sequence}, defined on $X_{\mathbf{a}}$. The scalar cardinality of $\tilde{A}$ is equal to that of the singleton set $\lbrace x \rbrace$, which is one.

Rule 5 implies the conservation of cardinality. The rationale is that during the construction process of the fuzzy set $\tilde{A}$, no element other than $x$ is involved. Consequently, only the classical set $\lbrace x \rbrace$ contributes to the cardinality of $\tilde{A}$. Thus, the cardinality of the fuzzy set $\tilde{A}$ remains unchanged compared to $\lbrace x \rbrace$.

In fact, Rule 5 is more fundamental than other rules, as the other rules can only be applied when a fuzzy set already exists. Therefore, the primary function of Rule 5 is to facilitate the initial construction of a fuzzy set from a classical one. Consequently, we assert that a variant of formula \eqref{rule4} is effective in the process of constructing a fuzzy set from a classical singleton set:\\

\noindent \textbf{Rule 6.} When constructing a fuzzy set $\tilde{A}$ from a singleton set $\lbrace x \rbrace$ using Rule 5, the following equations can be derived:
\begin{equation}\label{recursion1}
\mu_{\tilde{A}}\left(\lbrace x \rbrace^{(m)} \right) = 2^{\mu_{\tilde{A}}\left(\lbrace x \rbrace^{(m-1)} \right)} - 1,
\end{equation}
or equivalently,
\begin{equation}\label{recursion2}
\mu_{\tilde{A}}\left(\lbrace x \rbrace^{(m)} \right) = \log_2\left(\mu_{\tilde{A}}\left(\lbrace x \rbrace^{(m+1)} \right) + 1\right),
\end{equation}
where $m$ is an arbitrary integer.\\

Equations \eqref{recursion1} and \eqref{recursion2} involve a slight abuse of notation, as $\lbrace x \rbrace^{(m)}$ might not be included in the universe of the fuzzy set $\tilde{A}$ for certain values of $m$. However, this does not hinder our objective, which is to compute the membership value recursively using \eqref{recursion1} or \eqref{recursion2}. The notation $\lbrace x \rbrace^{(m)}$, even if not included in the universe, only appears in intermediate calculations.

We will now illustrate how to represent and calculate an arbitrary membership value using our rules. The specific methods for this calculation are detailed in the proof of the following theorem.

\begin{Theo}\label{core-theorem}
Given a binary sequence $\mathbf{a}$ defined by equation \eqref{a-binary-sequence}, there exists exactly one fuzzy set $\tilde A$ constructed from a single point set $\lbrace x \rbrace$ and defined on $X_{\mathbf{a}}$. Conversely, for any real number $0<w\leq1$, there exists a binary sequence $\mathbf{a}$ such that the fuzzy set $\tilde A$ constructed from $\lbrace x \rbrace$ and defined on $X_{\mathbf{a}}$ satisfies $\mu_{\tilde A}(x)=w$.
\end{Theo}

\begin{proof}
Consider a fuzzy set $\tilde{A}$ constructed from a single point set $\lbrace x \rbrace$. This fuzzy set $\tilde{A}$ is defined on $X_{\mathbf{a}}$, where $\mathbf{a} = (a_{m^*}, a_{m^*+1}, a_{m^*+2}, \dots), m^* \leq 0$. Additionally, $a_{m^*} \equiv 1$, $a_0 \equiv 1$, and the remaining $a_k$ values are either 0 or 1. To define recursive functions, we set
\begin{equation}
u_k(t) = \mu_{\tilde{A}}\left(\lbrace x \rbrace^{(k)}\right), \quad u_0(t) = \mu_{\tilde{A}}\left(x\right) = t \in [0,1].
\end{equation}
It is evident that $u_k(t)$ lies between 0 and 1 for any integer $k$.

\begin{enumerate}[(a)]
\item Given a binary sequence $\mathbf{a}$, the scalar cardinality of the corresponding fuzzy set $\tilde A$ can be computed using the following series:
\begin{equation}
G_{\mathbf{a}}(t):=\text{card}(\tilde A)=\sum_{k=m^*}^\infty a_k u_k(t).
\end{equation}
To demonstrate the first part of this theorem, we will divide the proof into three distinct cases.
\begin{enumerate}[I.]
\item 

 If $\mathbf{a}=(|)$, that is, $X_{\mathbf{a}}=\left\lbrace\{x\}^{(0)}\right\rbrace = \{x\}$,  then we have $G_{(|)}(t)=t$. Thus by Rule 5, $t=1$ is the only real number  satisfying the equation $G_{(|)}(t)=1$. Hence $\tilde A$ uniquely exists and degenerates into the classical set $\{x\}$. 
 
\item If $\mathbf{a}$ is finite, say, $\mathbf{a}= (a_{m^*},a_{m^*+1},\dots, a_N)$, $N\geq \max (m^*+1,0)$. Then the cardinality of fuzzy set $\tilde A$ defined on $X_{\mathbf{a}}$ is 
\begin{equation}
G_{\mathbf{a}}(t)=\sum_{k=m^*}^N a_k  u_k(t).
\end{equation}
It is easy to verify that
\begin{equation}\label{opposite-sign}
\big(G_{\mathbf{a}}(0)-1\big)\big(G_{\mathbf{a}}(1)-1\big)=-M<0,
\end{equation}
where $M$ is the number of $1$'s (not counting $a_0$) in binary sequence $\mathbf{a}$. Moreover, we have
\begin{equation}
\begin{aligned}
&\frac{\ud}{\ud t}\big(G_{\mathbf{a}}(t)-1\big)\\
=&\left\lbrace 
\begin{aligned}
&1+\sum_{k=1 }^{N }a_k(\ln  2)^k\cdot  2^{ u_0(t)+\cdots + u_{k-1}(t)}>0,\quad\text{if}\; m^* =0;\\[2mm]
&1+\sum_{k=m^* }^{-1 }a_k(\ln  2)^k\cdot 2^{- u_{-1}(t)-\cdots- u_{k}(t) }>0,\quad\text{if}\; m^* \leq -1, N=0;\\[2mm]
&\begin{aligned}
1&+\sum_{k=m^* }^{-1 }a_k(\ln  2)^k\cdot 2^{- u_{-1}(t)-\cdots- u_{k}(t) }\\
&+\sum_{k=1 }^{N }a_k(\ln  2)^k\cdot  2^{ u_0(t)+\cdots +\ u_{k-1}(t)}>0,
\end{aligned}
\quad\text{if}\; m^* \leq -1, N\geq 1.
\end{aligned}
\right.
\end{aligned}
\end{equation}

Hence $G_{\mathbf{a}}(t) - 1$ is strictly increasing on the interval $[0, 1]$. Combined with \eqref{opposite-sign}, it follows that $G_{\mathbf{a}}(t) - 1 = 0$ has exactly one root between 0 and 1, which corresponds to the value of $\mu_{\tilde{A}}(x)$.


\item If $\mathbf{a}$ is infinite, the cardinality of fuzzy set $\tilde A$ defined on $X_{\mathbf{a}}$ is given by the positive series
\begin{equation}
G_{\mathbf{a}}(t)=\sum_{k=0}^\infty a_{n_k} u_{n_k}(t),
\end{equation}

where $n_0=m^*$ and the subsequence $a_{n_k}$ consists of all $a_k=1$ in $\mathbf{a}$.

First, we need to prove that $G_{\mathbf{a}}(t)$ is convergent for any real $t$ in the open interval $(0, 1)$. For fixed $t \in (0, 1)$, since the positive $u_k(t)$ is strictly monotonically decreasing with respect to $k$, we have:
\begin{equation}
\lim_{k \to \infty} u_{n_k}(t) = \lim_{k \to \infty} u_k(t) = 0.
\end{equation}
Applying d'Alembert's ratio test, one computes the limit
\begin{equation}
\lim_{k \to \infty} \frac{u_{k+1}(t)}{u_{k}(t)} = \ln 2 < 1.
\end{equation}
Noting that $G_{\mathbf{a}}(0) \equiv 0$, thus the series $G_{\mathbf{a}}(t)$, as a subseries of the positive series $\sum_{k=m^*}^\infty u_k(t)$, converges for any binary sequence $\mathbf{a}$ and any real $t \in [0, 1)$.

Next we show that $G_{\mathbf{a}}(t)$ is continuous at any point $t_0$ in the open interval $(0,1)$. Without loss of generality, assume that $m^*<0$, and observe that
\begin{equation}
\begin{aligned}
&|u_k(t)-u_k(t_0)| =u_k'(c_k)\cdot |t-t_0|\\
&\left\lbrace
\begin{aligned}
=& |t-t_0|,\quad && \text{if }\; k=0;\\[2mm]
=&2^{u_0(c_k)+\cdots +u_{k-1}(c_k)}(\ln  2)^k \cdot   |t-t_0|,\quad && \text{if }\; k>0;\\[2mm]
=&2^{-u_{-1}(d_k)-\cdots -u_{k}(d_k)}(\ln  2)^k \cdot   |t-t_0|,\quad && \text{if }\; k<0,
\end{aligned}
\right.
\end{aligned}
\end{equation}

where both $c_k$ and $d_k$ are between $t$ and $t_0$. Let  $c_k^*=u_0(c_k)+\cdots +u_{k-1}(c_k)$, $d_k^*=-u_{-1}(d_k)-\cdots -u_{k}(d_k)$, then  we have
\begin{equation}
\begin{aligned}
&|G_{\mathbf{a}}(t)-G_{\mathbf{a}}(t_0)|\\
=&\left|\sum_{k=0}^\infty u_{n_k}(t) - \sum_{k=0}^\infty u_{n_k}(t_0)  \right|\\
\leq & \sum_{k=0}^\infty \left|  u_{n_k}(t)  - u_{n_k}(t_0) \right|\\
\leq & \sum_{k=m^*}^{-1} (\ln 2)^{k} \cdot 2^{d_k^*}\cdot  |t-t_0|+ |t-t_0|+\sum_{k=1}^\infty (\ln 2)^{k} \cdot 2^{c_{k}^*} \cdot   |t-t_0|\\
\leq &\left(P+\sum_{k=0}^\infty  (\ln 2)^{k} \cdot 2^{c^*}\right) \cdot   |t-t_0|\\
=&\left(P+\frac{ 2^{c^*}}{1-\ln 2}\right) \cdot |t-t_0|,
\end{aligned}
\end{equation}
where $P=\sum_{k=m^*}^{-1} (\ln 2)^{k} \cdot 2^{d_k^*}+1$ is bounded, and $c^*=\sup\lbrace c_{n_k}^* \rbrace$. Hence,  $G_{\mathbf{a}}(t)$ is continuous on $[0,1)$.

Furthermore, the derivative of $G_{\mathbf{a}}(t)$ 
\begin{equation}
\begin{aligned}
\frac{\ud}{\ud t}G_{\mathbf{a}}(t) = 1&+\sum_{k=m^* }^{-1 }a_k(\ln  2)^k\cdot 2^{-u_{-1}(t)-\cdots-u_{k}(t) }\\
&+\sum_{k=1 }^{\infty }a_k(\ln  2)^k\cdot  2^{u_0(t)+\cdots +u_{k-1}(t)}>0.
\end{aligned}
\end{equation}
Thus, $G_{\mathbf{a}}(t)$ is strictly increasing on the interval $(0,1)$.

Now, because $G_{\mathbf{a}}(0) - 1 \equiv -1$ and $G_{\mathbf{a}}(t) - 1 \to +\infty$ as $t \to 1$, we can confirm that for any infinite binary sequence $\mathbf{a}$, the equation $G_{\mathbf{a}}(t) - 1 = 0$ has exactly one root between $0$ and $1$, which corresponds to the value of $\mu_{\tilde{A}}(x)$.

\end{enumerate}

\item Conversely, we provide an algorithm to find the binary sequence $\mathbf{a}$ for any given real number $w \in (0, 1]$ such that the fuzzy set $\tilde{A}$ generated from $\lbrace x \rbrace$ and defined on $X_{\mathbf{a}}$ satisfies $\mu_{\tilde{A}}(x) = w$.

If $w = 1$, the special binary sequence $(|)$ is what we are seeking.

If $0 < w < 1$, the following algorithm generates the binary sequence $\mathbf{a}= (a_{m^*},a_{m^*+1},a_{m^*+2},\dots)$.
\begin{enumerate}[I.]
\item Find the initial nonzero term. \\
Define the function $s_0(k) = u_k(w) + w - 1$, $k \neq 0$. For any fixed $w \in (0, 1)$, since $s_0(k)$ is strictly monotonically decreasing with respect to $k$ with $\lim_{k \to -\infty} s_0(k) = w > 0$ and $\lim_{k \to \infty} s_0(k) = w - 1 < 0$, there exists an integer $n_0$ such that $s_0(k) > 0$ if $k < n_0$ and $s_0(k) \leq 0$ if $k \geq n_0$. We set $m^* = n_0$ if $n_0 < 0$ and $m^* = 0$ if $n_0 > 0$. If $s_0(n_0) = 0$, this algorithm ends with a finite $\mathbf{a} = (a_{m^*}, \dots, |)$ if $n_0 < 0$ or $\mathbf{a} = (|, \dots, a_{n_0})$ if $n_0 > 0$. In both cases, except $a_{n_0} = 1$ and $a_0 = 1$, all other $a_k$'s (if they exist) in $\mathbf{a}$ are set to $0$. Otherwise, if $s_0(n_0) < 0$, proceed to Step II.

\item Find the second  nonzero term.\\
Define the function $s_1(k) = s_0(n_0) + u_k(w)$, $k > n_0$ and $k \neq 0$. For any fixed $w \in (0, 1)$, $s_1(k)$ is strictly monotonically decreasing with respect to $k$ since $u_k(w)$ is strictly monotonically decreasing with respect to $k$. Because:
\begin{equation}
s_1(n_0 + 1) - s_0(n_0 - 1) = 2^{2^{u_{n_{\scalebox{0.4}{0}}-1}(w)} - 1} + 2^{u_{n_{\scalebox{0.4}{0}}-1}(w)} - u_{n_0-1}(w) - 2 > 0
\end{equation}
holds for any fixed $w \in (0, 1)$, we have:
\begin{equation}
s_1(n_0 + 1) > s_0(n_0 - 1) > 0.
\end{equation}

Together with $\lim_{k \to \infty} s_1(k) = s_0(n_0) < 0$, there exists an integer $n_1$ such that $s_1(k) > 0$ if $n_0 < k < n_1$ and $s_1(k) \leq 0$ if $k \geq n_1$. Thus, we set $a_k = 0$ for $n_0 < k < n_1$ and $a_{n_1} = 1$, which is the second (not counting $a_0$) nonzero term in $\mathbf{a}$. If $s_1(n_1) = 0$, this algorithm ends with a finite $\mathbf{a}$, which takes one of the following forms depending on the position of $a_0$:
\begin{enumerate}
\item $(|, \dots, a_{n_0}, \dots, a_{n_1})$,
\item $(a_{n_0}, \dots, |, \dots, a_{n_1})$,
\item $(a_{n_0}, \dots, a_{n_1}, \dots, |)$.
\end{enumerate}
Except for $a_{n_0} = 1$, $a_{n_1} = 1$, and $a_0 = 1$, all other $a_k$'s (if they exist) in $\mathbf{a}$ are set to 0. Otherwise, if $s_1(n_1) < 0$, proceed to Step III.

\item Find the $j$-th, $j\geq 2$, nonzero term  $a_{n_j}$.\\

Suppose that the $(j-1)$-th (not counting $a_0$) nonzero term in $\mathbf{a}$ is $a_{n_{j-1}}$ and the functions $s_j(k)$ are defined recursively as:
\begin{equation}
s_j(k) = s_{j-1}(n_{j-1}) + u_k(w), \quad k > n_{j-1}, k \neq 0,
\end{equation}
satisfying that $s_{j-1}(k) > 0$ if $n_{j-2} < k < n_{j-1}$ and $s_{j-1}(k) < 0$ if $k \geq n_{j-1}$. For any fixed $w \in (0, 1)$, $s_j(k)$ is strictly monotonically decreasing with respect to $k$ since $u_k(w)$ is strictly monotonically decreasing with respect to $k$. To lighten notations, let $\xi$ denote $u_{n_{j-1}-1}(w)$.
Because:
\begin{equation}
s_j(n_{j-1} + 1) - s_{j-1}(n_{j-1} - 1) = 2^{2^\xi - 1} + 2^\xi - \xi - 2 > 0,
\end{equation}
holds for any $\xi \in (0, 1)$, we have:
\begin{equation}
s_j(n_{j-1} + 1) > s_{j-1}(n_{j-1} - 1) > 0.
\end{equation}

Together with:
\begin{equation}
\lim_{k \to \infty} s_j(k) = s_{j-1}(n_{j-1}) < 0,
\end{equation}
there exists an integer $n_j$ such that $s_j(k) > 0$ if $n_{j-1} < k < n_j$ and $s_j(k) \leq 0$ if $k \geq n_j$. Thus, we set $a_k = 0$ for $n_{j-1} < k < n_j$ and $a_{n_j} = 1$, which is the $j$-th (not counting $a_0$) nonzero term in $\mathbf{a}$. If $s_j(n_j) = 0$, this algorithm ends with a finite $\mathbf{a}$. Otherwise, if $s_j(n_j) < 0$, repeat this process to find the $(j+1)$-th nonzero term $a_{n_{j+1}}$ in $\mathbf{a}$.

Clearly, if $\mathbf{a}$ is infinite by this algorithm, then $w$ satisfies the equation $G_{\mathbf{a}}(w) = 1$, and this completes our proof.

\end{enumerate}
\end{enumerate}
\end{proof}

\section{Numerical Examples}  
In this section, we will first provide an illustrative example to demonstrate how membership values can be calculated using our construction rules. Next, we will present a specific numerical example to verify \textbf{Theorem \ref{cardinality-power-set}}, illustrating the correctness of the fuzzy cardinality formula \eqref{cardinality-powerset} involving the power set. Then, we will present several numerical examples to explain how to use \textbf{Theorem \ref{core-theorem}} to achieve and represent any desired membership value.

\begin{Exa}
Consider a fuzzy set  $\tilde A$ defined on the universe $X=\{x_1,x_2,x_3,x_4\}$ as:
\begin{equation}
\tilde A=  \frac{0.2}{x_1}+\frac{0.3}{x_2}+\frac{0.5}{x_3}+ \frac{1}{x_4} .
\end{equation}
We can construct a new fuzzy set $\tilde B$  on the universe 
\begin{equation}
Y=\Bigg\{ \bigg\{\varnothing, x_1 \bigg\},  \bigg\{\{x_2\},\{x_3\}\bigg\}, \bigg\{x_1,\Big\{x_2,\big\{x_3,\{x_4\}\big\}\Big\}\bigg\}\Bigg\}.
\end{equation}
from $\tilde A$ by applying formula (\ref{rule4}) repeatedly. The resulting values of $\mu_{\tilde B}(y)$ for each $y\in Y$ are:
\begin{equation}
\mu_{\tilde B}\big(\{\varnothing, x_1\}\big)=\left(2^1-1  \right)\left(2^{\mu_{\tilde A}(x_1)}-1  \right)=2^{0.2}-1 \approx 0.1487;
\end{equation}

\begin{equation}
\begin{aligned}
&\mu_{\tilde B}\Big(\big\{\{x_2\},\{x_3\}\big\}\Big)\\
=&\left(2^{\left(2^{\mu_{\tilde A}(x_2)}-1 \right)}-1  \right)\left(2^{\left(2^{\mu_{\tilde A}(x_3)}-1 \right)}-1  \right)\\
=&\left(2^{\left(2^{0.3}-1 \right)}-1  \right)\left(2^{\left(2^{0.5}-1 \right)}-1  \right)\\
\approx&0.0364;
\end{aligned}
\end{equation}

\begin{equation}
\begin{aligned}
&\mu_{\tilde B}\Bigg( \bigg\{x_1,\Big\{x_2,\big\{x_3,\{x_4\}\big\}\Big\}\bigg\}\Bigg)\\
=&\left(2^{\mu_{\tilde A}(x_1)}-1  \right)\left(    2^{\left(2^{\mu_{\tilde A}(x_2)}-1  \right)\left( 2^{\left(2^{\mu_{\tilde A}(x_3)}-1  \right)\left(   2^{\left(2^{\mu_{\tilde A}(x_4)}-1  \right)}-1\right)}-1  \right)}-1    \right)\\
=&\left(2^{0.2}-1  \right)\left(    2^{\left(2^{0.3}-1  \right)\left( 2^{\left(2^{0.5}-1  \right)}-1  \right)}-1    \right)\\
\approx&0.0081.
\end{aligned}
\end{equation}

Therefore, we can represent $\tilde B$ as:

\begin{equation}
\tilde B= \frac{0.1487}{\{\varnothing, x_1\}}+\frac{0.0364}{\{\{x_2\},\{x_3\}\}}+ \frac{0.0081}{\{x_1,\{x_2,\{x_3,\{x_4\}\}\}\} }.
\end{equation}
\end{Exa}

\begin{Exa}
Consider the fuzzy set  $\tilde A$ defined on the universe $X=\{x_1,x_2,x_3\}$ as:
\begin{equation}
\tilde A=  \frac{0.2}{x_1}+\frac{0.3}{x_2}+\frac{0.5}{x_3} .
\end{equation}
Construct a fuzzy set $\tilde B$ on the power set $\mathcal{P}(X)$  of  $X$, and verify \textbf{Theorem \ref{cardinality-power-set}}.
\end{Exa}

According to the definition, the cardinality of the fuzzy set $\tilde A$ is calculated as:
\begin{equation}
\text{card}(\tilde A)=  0.2+0.3+0.5=1.
\end{equation}
We proceed to construct the fuzzy set $\tilde B$ on $\mathcal{P}(X)$
 as follows:
\begin{equation}
\begin{aligned}
\tilde B&=\displaystyle\frac{1}{ \varnothing }+\frac{ \mu_{\tilde B}(\{x_1\})}{ \{x_1\} }+\frac{ \mu_{\tilde B}(\{x_2\})}{ \{x_2\} }+\frac{ \mu_{\tilde B}(\{x_3\})}{ \{x_3\} }+\frac{ \mu_{\tilde B}(\{x_1,x_2\})}{ \{x_1,x_2\} }\\
&+\frac{ \mu_{\tilde B}(\{x_2,x_3\})}{ \{x_2,x_3\} }+\frac{ \mu_{\tilde B}(\{x_1,x_3\})}{ \{x_1,x_3\} }+\frac{ \mu_{\tilde B}(\{x_1,x_2,x_3\})}{ \{x_1,x_2,x_3\} }.
\end{aligned}
\end{equation}
Substituting the membership degrees from $\tilde A$, we have:
\begin{equation}
\mu_{\tilde B}(\{x_1\}) = 2^{\mu_{\tilde A}(x_1)}-1=2^{0.2}-1,
\end{equation}
\begin{equation}
\mu_{\tilde B}(\{x_2\}) = 2^{\mu_{\tilde A}(x_2)}-1=2^{0.3}-1,
\end{equation}
\begin{equation}
\mu_{\tilde B}(\{x_3\}) = 2^{\mu_{\tilde A}(x_3)}-1=2^{0.5}-1.
\end{equation}
For the two-element subsets:
\begin{equation}
\mu_{\tilde B}(\{x_1,x_2\}) = \left(2^{\mu_{\tilde A}(x_1)}-1\right)\left(2^{\mu_{\tilde A}(x_2)}-1\right)=2^{0.5}-2^{0.3}-2^{0.2}+1,
\end{equation}
\begin{equation}
\mu_{\tilde B}(\{x_2,x_3\}) = \left(2^{\mu_{\tilde A}(x_2)}-1\right)\left(2^{\mu_{\tilde A}(x_3)}-1\right)=2^{0.8}-2^{0.5}-2^{0.3}+1,
\end{equation}
\begin{equation}
\mu_{\tilde B}(\{x_1,x_3\}) = \left(2^{\mu_{\tilde A}(x_1)}-1\right)\left(2^{\mu_{\tilde A}(x_3)}-1\right)=2^{0.7}-2^{0.5}-2^{0.2}+1,
\end{equation}
For the three-element subset:
\begin{equation}
\begin{aligned}
\mu_{\tilde B}(\{x_1,x_2,x_3\}) &= \left(2^{\mu_{\tilde A}(x_1)}-1\right)\left(2^{\mu_{\tilde A}(x_2)}-1\right)\left(2^{\mu_{\tilde A}(x_3)}-1\right)\\
&=1+2^{0.2}+2^{0.3}-2^{0.7}-2^{0.8}.
\end{aligned}
\end{equation}
Next, we compute the cardinality of 
$\tilde B$:
\begin{equation}
\begin{aligned}
\text{card}(\tilde B)= &1+2^{0.2}-1+2^{0.3}-1+2^{0.5}-1+2^{0.5}-2^{0.3}-2^{0.2}+1+2^{0.8}-2^{0.5}-2^{0.3}+1\\
&+2^{0.7}-2^{0.5}-2^{0.2}+1+1+2^{0.2}+2^{0.3}-2^{0.7}-2^{0.8}\\
&=2 =2^{\text{card}(\tilde A)}.
\end{aligned}
\end{equation}
This confirms \textbf{Theorem \ref{cardinality-power-set}}.

\begin{Exa}
Given two finite binary sequences, $\mathbf{a} = (10|01)$ and $\mathbf{b} = (|01001)$, we aim to find the corresponding fuzzy sets, $\tilde{A}$ and $\tilde{B}$, both of which are constructed from a single point set $\lbrace x \rbrace$ and defined on $X_{\mathbf{a}}$ and $X_{\mathbf{b}}$, respectively.
\end{Exa}

To find the values of $\mu_{\tilde A}(x)$ and $\mu_{\tilde B}(x)$, we need to solve the following two equations:

\begin{equation}
\log_2(\log_2(u_{a}+1)+1)+u_a+2^{2^{u_a}-1}-1=1,
\end{equation}

\begin{equation}
u_b+2^{2^{u_b}-1}-1+2^{2^{2^ {2^{2^{u_b}-1}-1} -1} -1}-1=1,
\end{equation}

where $u_a$ and $u_b$ represent the values of $\mu_{\tilde A}(x)$ and $\mu_{\tilde B}(x)$, respectively.

Solving these equations gives us $\mu_{\tilde A}(x)\approx 0.3222$ and $\mu_{\tilde B}(x)\approx 0.5087$. Therefore, the two fuzzy sets are:

\begin{equation}
\begin{split}
\tilde A &= \frac{\log_2(\log_2(u_{a}+1)+1)}{\lbrace x \rbrace ^{(-2)}}+\frac{u_{a}}{x}+\frac{2^{2^{u_a}-1}-1}{\lbrace x \rbrace ^{(2)}}\\
&=\frac{0.4884}{\lbrace x \rbrace ^{(-2)}}+\frac{0.3222}{x}+\frac{0.1894}{\lbrace x \rbrace ^{(2)}}
\end{split}
\end{equation}

and

\begin{equation}
\begin{split}
\tilde B &= \frac{u_{b}}{x}+\frac{ 2^{2^{u_b}-1}-1}{\lbrace x \rbrace ^{(2)}}+\frac{2^{2^{2^ {2^{2^{u_b}-1}-1} -1} -1}-1 }{\lbrace x \rbrace ^{(5)}}\\
&=\frac{0.5087}{x}+\frac{ 0.3405}{\lbrace x \rbrace ^{(2)}}+\frac{0.1508}{\lbrace x \rbrace ^{(5)}}.
\end{split}
\end{equation}

It is easy to verify that both $\tilde A$ and $\tilde B$ have a cardinality of 1.

\begin{Exa}
Given two real numbers, $w_a$ and $w_b$, where $w_a=0.3$ and $w_b=0.8$, we need to find binary sequences $\mathbf{a}$ and $\mathbf{b}$ that correspond to fuzzy sets $\tilde A$ and $\tilde B$, respectively. These fuzzy sets are constructed from the set $\lbrace x \rbrace$ and are defined on $X_{\mathbf{a}}$ and $X_{\mathbf{b}}$, respectively. Furthermore, we want the fuzzy sets to satisfy $\mu_{\tilde A}(x)=w_a$ and $\mu_{\tilde B}(x)=w_b$.
\end{Exa}

To determine the initial nonzero term of $\mathbf{a}$, we first evaluate the expression:
\begin{equation}
\log_2(\log_2(\log_2(\log_2(\log_2(0.3 + 1) + 1) + 1) + 1) + 0.3 - 1\approx 0.0061>0
\end{equation}
and
\begin{equation}
\log_2(\log_2(\log_2(\log_2(0.3 + 1) + 1) + 1) + 1) + 0.3 - 1\approx -0.0686<0,
\end{equation}
which indicates that the initial nonzero term in $\mathbf{a}$ is $a_{-4}$. We repeat the same steps for subsequent terms to obtain $a_{5}$, $a_{15}$, $a_{20}$, and so on. Therefore, the binary sequence $\mathbf{a}$ can be expanded as:
\begin{equation}
\mathbf{a} = (1,000, |,000,010,000,000,001,000,010,\dots)
\end{equation}
Similarly, for $w_b=0.8$, the binary sequence $\mathbf{b}$ can be expanded as:
\begin{equation}
\mathbf{b} = (|000,000,001,000,100,010,000,000,000,010,\dots)
\end{equation}

\section{Conclusion}
In this contribution, we propose novel methods for constructing new fuzzy sets from existing ones or classical sets. These methods consist of six interrelated rules that are compatible with classical sets. Our methods are objective, meaning that both the essential problems of constructing new fuzzy sets and achieving the values of membership functions are obtained through objective calculations or by solving objectively constructed equations, without any subjective assumptions.

The first result of our methods generalizes an important property of the power set from classical settings to fuzzy settings. Additionally, the second result reveals deep connections between fuzzy sets and binary sequences, along with a new algorithm. By studying the corresponding binary sequences, it becomes possible to gain a more profound knowledge of fuzzy sets. This will be the subject of our future research.

\section*{Acknowledgment}
This study received funding from the Natural Science Foundation of Southwest University of Science and Technology under Grant 21zx7115.


\bibliographystyle{elsarticle-num}

\bibliography{FSDPSU}

%

\end{document}